\documentclass[a4paper,  leqno,  myheadings,  11pt,  twoside]{article}
\usepackage{mathrsfs}

\usepackage{amsmath,  amsthm,  amssymb,  amscd,  amsxtra, graphicx}
\usepackage{latexsym,  amsfonts}

\DeclareMathOperator*{\esssup}{esssup}

\def\qc{quasi\-conformal }
\def\po{Poincar\'{e}}

\def\md{\de}
\def\mc{\mathbb{C}}
\def\mr{\mathbb{R}}

\def\T{Teich\-m\"ul\-ler }

\def\bmu{\emu_Z}

\def\nmu{\|\mu\|_\infty}

\def\nmu{\|\mu\|_\infty}

\def\wt{\widetilde}
\def\vp{\varphi}

\def\vp{\varphi}

\def\limn{\lim_{n\to\infty}Re}

\def\pa{\partial}

\def\ov{\overline}

\def\de{\Delta}
\def\q1s{Q^1(S)}

\def\ts{T(S)}

\def\zs{Z(S)}
\def\azs{AZ(S)}
\def\azde{AZ(\de)}

\def\ats{AT(S)}

\def\qds{Q^1_d(S)}
\def\qdde{Q^1_d(\de)}

\def\atde{AT(\de)}
\def\dat{d_{AT}}
\def\oo{[[0]]}
\def\oaz{\oo_{AZ}}

\def\emu{[\mu]}

\def\emb{[\mu]_Z}

\def\azmu{[[\mu]]_{AZ}}

\def\mut{[[\mu]]}

 \makeatletter
\@addtoreset{equation}{section}

\newtheorem{theorem}{Theorem}
\newtheorem{cor}{Corollary}
\newtheorem{lemma}{Lemma}[section]

\begin{document}

\title{\bf{Geodesic disks
in asymptotic \T space
 }}
\author{GUOWU YAO
}
\date{}
\maketitle
\begin{abstract}\noindent
Let $S$ be a hyperbolic Riemann surface.
 In a finite-dimensional \T space $\ts$, it is still an open problem whether the geodesic disk passing through two points is unique. In an infinite-dimensional \T space it is also unclear how many geodesic disks pass through a Strebel point and the basepoint while we know that there are always geodesic disks passing through a non-Strebel point and the basepoint. In this paper, we answer the problem arising in the universal asymptotic \T space and prove that there are always infinitely many geodesic disks passing through two points.

\end{abstract}
\renewcommand{\thefootnote}{}
\footnote{The   work  was
supported by the National Natural Science Foundation of China
(Grant No. 11271216).}


\begin{centering}\section{\!\!\!\!\!{. }
Introduction}\label{S:intr}\end{centering}

Let $S$ be a hyperbolic Riemann surface,  that is,  it is  covered
by a holomorphic map: $\varpi: \de\to S$,  where $\de=\{|z|<1\}$ is the
open unit disk. Let $T(S)$ be the \T space of $S$.  A quotient space of the  \T space $T(S)$, called the asymptotic \T
space and denoted by $AT(S)$, was introduced by Gardiner and
Sullivan (see~\cite{GS} for $S=\de$ and by Earle, Gardiner and
Lakic for arbitrary hyperbolic
$S$~\cite{EGL0,EGL,GL}).

 $\ats$ is interesting only when $\ts$ is infinite dimensional,
which occurs when $S$ has border or when $S$ has infinite
topological type, otherwise, $\ats$ consists of just one point. In
recent years, the asymptotic space $\ats$ and its tangent space are
extensively studied, for examples, one can refer to
\cite{EGL0,EGL,EMS, Fuj, Mats, Mi, Yao}.

We shall use some geometric terminologies  adapted from \cite{Bus}
by Busemann. Let $X$ and $Y$ be metric spaces. An isometry of $X$
into $Y$ is a distance preserving map. A straight line in $Y$ is a
(necessarily closed) subset $L$ that is  an isometric image of the
real line $\mathbb{R}$.   A geodesic in $Y$ is an isometric image of
a non-trivial compact interval of $\mathbb{R}$. Its endpoints are
the images of the endpoints of the interval, and we say that the
geodesic joins its endpoints.

 Geodesics play an important role in the theory
 of \T spaces.
In an finite-dimensional \T space $\ts$, there is always a unique
geodesic connecting  two points. The geometry  is
substantially different in an infinite dimensional \T space (see
\cite{ELi,Li, Li2, Li3,Ta}). Generally,  for  a Strebel point,
   there is a unique geodesic connecting it and the basepoint. The situation on geodesics in the asymptotic \T space is still unclear while Fan \cite{Fan} gave certain examples to
show the nonuniqueness of geodesics in asymptotic spaces.

Another important role in the \T theory is the geodesic disk which is defined as the  image of an isometric embedding from $\de$ into the \T space $T(S)$ with respect to the hyperbolic metric on $\de$ and the \T metric on $T(S)$ respectively. Using holomorphic motion, Earle et al. considered holomorphic geodesic disks containing two points in \cite{EKK}. In \cite{Li4}, Li proved that there are always infinitely many geodesic disks passing through a non-Strebel point and the basepoint.  So far, we know little information on how many  geodesic disks passing through a Strebel point and the basepoint. It is even unknown whether the geodesic disk passing through two points is unique in a finite-dimensional \T space.

The motivation of the paper is to investigate  geodesic disks in the asymptotic \T space. We characterize the nonuniqueness of geodesic disks in the universal asymptotic space $\atde$ completely. That is,
\begin{theorem}\label{Th:disks}
In the universal asymptotic \T space $\atde$, there are always  infinitely many
geodesic disks containing   two points.
\end{theorem}

This paper is organized as follows. In Section \ref{S:prel}, we introduce some basic notion in  the \T space theory. In Section \ref{S:inequality}, an  infinitesimal inequality of the asymptotic \T metric is founded. Theorem \ref{Th:disks} is proved in Section \ref{S:disk}. A parallel version of Theorem \ref{Th:disks} in the infinitesimal setting is obtained in the last section.

The method used here can also be used to deal with some more general cases. However, there are some difficulties in solving the problem in all cases.

\begin{centering}\section{\!\!\!\!\!{. }Some Preliminaries}\label{S:prel}\end{centering}

\subsection{\T space and asymptotic \T space}
Let $S$ be a Riemann surface of topological type.  The \T space
$\ts$ is the space of equivalence classes of \qc maps $f$ from $S$
to a variable Riemann surface $f(S)$. Two \qc maps $f$ from $S$ to
$f(S)$ and $g$ from $S$ to $g(S)$ are equivalent if there is a
conformal map $c$ from $f(S)$ onto $g(S)$ and a homotopy through \qc
maps $h_t$ mapping $S$ onto $g(S)$ such that $h_0=c\circ f$, $h_1=g$
and $h_t(p)=c\circ f(p)=g(p)$ for every $t\in [0,1]$ and every $p$
in the ideal boundary of $S$. Denote by  $[f]$  the \T equivalence
class of $f$; also sometimes denote the equivalence class by $\emu$
where $\mu$ is the Beltrami differential of $f$.

The asymptotic \T space is the space of a larger equivalence classes. The definition of the new equivalence classes
 is exactly
the same as the previous definition with one exception; the word
\emph{conformal } is replaced by \emph{asymptotically conformal}. A
\qc map $f$ is asymptotically conformal if for every $\epsilon>0$,
there is a compact subset $E$ of $S$, such that the dilatation of
$f$ outside of $E$ is less than $1+\epsilon$. Accordingly, denote by
$[[f]]$ or $\mut$ the asymptotic equivalence class of $f$.

Denote by $Bel(S)$ the Banach space of Beltrami differentials
$\mu=\mu(z){d\bar z}/dz$ on $S$ with finite $L^{\infty}$-norm and by
$M(S)$ the open unit ball in $Bel(S)$.

For $\mu\in M(S)$, define
\begin{equation*}
k_0(\emu)=\inf\{\|\nu\|_\infty:\,\nu\in\emu\}.
\end{equation*}

Define $h^*(\mu)$ to be the infimum over all
 compact subsets $E$ contained in $S$ of the essential supremum
 norm of the Beltrami differential $\mu(z)$ as $z$ varies over
 $S\backslash E$ and  $h(\emu)$ to be the infimum of
 $h^*(\nu)$ taken over all representatives $\nu$ of the class
 $\emu$.  It is obvious that $h(\emu)\leq
 k_0(\emu)$. Following \cite{ELi}, $\emu$ is called
 a Strebel point if $h(\emu)<k_0(\tau)$; otherwise, $\tau$ is called
 a non-Strebel point.

Put
\begin{equation*}
h(\mut)=\inf\{h^*(\nu):\,\nu\in\mut\}.
\end{equation*}
 We say that $\mu$ is extremal  in $\emu$
  if $\nmu=k_0(\emu)$ and $\mu$ is asymptotically extremal if $h^*(\mu)=h(\mut)$. The relation $h(\emu)=h(\mut)$
   is due to the definition.

The \T metric $d_{T}$ between two points $\tau, \sigma \in \ts$ is
defined as follows:
\begin{equation*}
d_{T}(\tau, \sigma)=\frac{1}{2}\inf_{\mu\in\tau,\ \nu\in
\sigma}\log\frac{1+\|(\mu-\nu)/(1-\bar\nu\mu)\|_\infty}{1-\|(\mu-\nu)/(1-\bar\nu\mu)\|_\infty}.
\end{equation*}
The asymptotic \T metric $d_{AT}$ between two points $\wt\tau,
\wt\sigma \in \ats$ is defined by
\begin{equation*}
d_{AT}(\wt\tau,\wt \sigma)=\frac{1}{2}\inf_{\mu\in\wt\tau,\ \nu\in
\wt\sigma}\log\frac{1+\|(\mu-\nu)/(1-\bar\nu\mu)\|_\infty}{1-\|(\mu-\nu)/(1-\bar\nu\mu)\|_\infty}.
\end{equation*}
In particular, the  distance between $\mut$ and the basepoint $\oo$ is
\begin{equation*}
d_{AT}(\mut, \oo)=\frac{1}{2}\log H(\mut),\;\text {where\; }H(\mut)=\frac{1+h(\mut)}{1-h(\mut)}.
\end{equation*}

\subsection{Tangent spaces to \T space and asymptotic \T space}

The cotangent space to $\ts$ at the basepoint is the Banach space
$Q(S)$ of integrable holomorphic quadratic differentials $\vp$ on $S$ with
$L^1-$norm
\begin{equation*}
\|\vp\|=\iint_{S}|\vp(z)|\, dxdy<\infty.
 \end{equation*}    In what follows,  let $\q1s$
denote the unit sphere of $Q(S)$. Moreover, let $\qds$ denote the
set of all degenerating sequence $\{\vp_n\}\subset \q1s$. By
definition, a sequence $\{\vp_n\}$ is called degenerating if it
converges to 0 uniformly on compact subset of $S$.

Two Beltrami differentials $\mu$ and $\nu$ in $Bel(S)$ are said to
be infinitesimally equivalent if
\begin{equation*}\iint_S(\mu-\nu)\vp \, dxdy=0,  \text{
for any } \vp\in Q(S).
\end{equation*}
The tangent space $\zs$ of $\ts$ at the basepoint is defined as the
set of the quotient space of $Bel(S)$ under the equivalence
relations.  Denote by $\emb$ the equivalence class of $\mu$ in
$\zs$. The  set of all  Beltrami differentials equivalent to zero is
called the $N-$class in $Bel(S)$.

$\zs$ is a Banach space and actually \cite{GL}  its standard
sup-norm satisfies
\begin{equation*}
\|\emb\|:=\sup_{\vp\in \q1s}Re\iint_S \mu\vp \,
dxdy=\inf\{\|\nu\|_\infty:\,\nu\in\emb\}.\end{equation*}

Two Beltrami differentials $\mu$ and $\nu$ in $Bel(S)$ are said to
be infinitesimally asymptotically equivalent if
\begin{equation*}\sup_{\qds}\limsup_{n\to\infty} Re\iint_S(\mu-\nu)\vp_n \, dxdy=0,
\end{equation*}
where the first  $supremum$ is taken when  $\{\vp_n\}$ varies over
$\qds$.

 The tangent space $\azs$ of $\ats$ at the basepoint is
defined as the set of the quotient space of $Bel(S)$ under the asymptotic
equivalence relation.  Denote by $\azmu$ the equivalence class of
$\mu$ in $\azs$. The  set of all  Beltrami differentials equivalent
to zero is called the $Z_0-$class in $Bel(S)$.

Define  $b(\bmu)$ to be  the infimum over all elements in the
equivalence class $\bmu$  of the quantity $b^*(\nu)$. Here
$b^*(\nu)$ is the infimum over all compact subsets $E$ contained in
$S$ of the essential supremum of the the Beltrami differential $\nu$
as $z$ varies over $S-E$. It is obvious that $b^*(\mu)\leq\|\bmu\|
$. $\bmu$ is called
 an infinitesimal Strebel point if $b(\bmu) <\|\emb\|$. We say a
 Beltrami differential $\mu\in Bel(S)$ vanishing at infinity if
 $b^*(\mu)=0$.

Put
\begin{equation*}
b(\azmu)=\inf\{b^*(\nu):\,\nu\in\azmu\}.
\end{equation*}

 We say that $\mu$ is (infinitesimally) extremal
  if $\nmu=\|\emb\|$ and $\mu$ is (infinitesimally) asymptotically extremal if $b^*(\mu)=b(\azmu)$.
 We also have $b(\emb)=b(\azmu)$ \cite{GL}.

$\azs$ is a Banach space and   its  standard infinitesimal
asymptotic norm satisfies (see \cite{GL})
\begin{equation*}
\|\azmu\|:=\sup_{\qds}\limsup_{n\to\infty} Re\iint_S \mu\vp_n \,
dxdy=\inf\{\|\nu\|_\infty:\,\nu\in\azmu\}=b(\azmu).\end{equation*}

\subsection{Substantial boundary points and Hamilton sequence}

Now we define the notion of boundary dilatation of a \qc mapping at
a boundary point. For a Riemann surface, the meaning of what is a
boundary point can be problematic. However, if  $S$ can be embedded into a larger surface $\wt S$ such that the closure of
$S$ in $\wt S$ is compact, then it is possible to define the
boundary dilatation. From now on, we assume that $S$ is such a surface.

Let $p$ be a point on $\pa S$ and let $\mu\in Bel(S)$. Define
 \begin{equation*}
 h^*_p(\mu)=\inf\{\esssup_{z\in U\bigcap S}|\mu(z)|:\;U \text{ is an open
 neighborhood
 in } \wt{S} \text{ containing }p\}
 \end{equation*}
 to be the boundary dilatations of $\mu$ at  $p$.
If $\mu\in M(S)$, define
 \begin{equation*}
 h_p([\mu])=\inf\{h^*_p(\nu):\;\nu\in[\mu]\}
 \end{equation*}
to be the boundary dilatations $[\mu]$ at  $p$. For a general $\mu\in
Bel(S)$, define
 \begin{equation*}
 b_p([\mu]_Z)=\inf\{h^*_p(\nu):\;\nu\in \emb\}
 \end{equation*}
to be the boundary dilatations of $\emb$ at  $p$. If we define the
quantities
 \begin{equation*}
 h_p(\mut)=\inf\{h^*_p(\nu):\;\nu\in \mut\},\quad b_p(\azmu)=\inf\{h^*_p(\nu):\;\nu\in
 \azmu\},
 \end{equation*}
then  $h_p(\emu)=h_p(\mut)$ and $b_p(\emb)=b_p(\azmu)$. In particular,  Lakic \cite{La} proved that when $S$ is a plane domain,
\begin{equation*}
h(\mut)=\max_{p\in \pa S} h_p(\mut),\quad b(\azmu)=\max_{p\in \pa S}
 b_p(\azmu).
 \end{equation*}

 As is well
known, $\mu$ is extremal if and only if
  it  has a  so-called Hamilton sequence, namely, a sequence
$\{\psi_n\}\subset \q1s$, such that
\begin{equation}
\lim_{n\to\infty}Re\iint_S \mu\psi_n(z)dxdy=\nmu.
\end{equation}
Similarly, by Theorem 8 on page 281 in \cite{GL}, $\mu$ is asymptotically extremal if and only if
  it  has an asymptotic Hamilton sequence, namely, a degenerating  sequence
$\{\psi_n\}\subset \q1s$, such that
\begin{equation}
\lim_{n\to\infty}Re\iint_S \mu\psi_n(z)dxdy=h^*(\mu).
\end{equation}

Now, we assume that $S$ is a plane domain with two or more boundary points. Then, the following lemma derives from  Theorem 6 on page 333 in ~\cite{GL}:
 \begin{lemma}\label{Th:atslem1}
The following three conditions are equivalent for every boundary point
$p$ of $S$ and every asymptotic  or infinitesimal asymptotic extremal representative $\mu$:
\\(1) $h(\emu)=h_p(\emu)$ (equivalently, $h(\mut)=h_p(\mut)$),\\
(2) $b(\emu)=b_p(\emu)$   (equivalently,  $b(\azmu)=b_p(\azmu)$),\\
(3) there exists an asymptotic Hamilton sequence for $\mu $ degenerating towards
$p$, i.e., a sequence
$\{\psi_n\}\subset\q1s$   converging  uniformly to 0 on compact subsets of
$S\backslash\{p\}$, such that
\begin{equation}
\lim_{n\to\infty}Re\iint_S \mu\psi_n(z)dxdy=h_p^*(\mu).
\end{equation}
 \end{lemma}
 If one  of three conditions in the  lemma holds at some $p\in
\pa S$, we call $p$ is a substantial boundary point for $\mut$ (or
$\emu$) and $\azmu$ (or $\emb$), respectively.

\begin{centering}\section{\!\!\!\!\!{. }An infinitesimal inequality for asymptotic \T metric }\label{S:inequality}\end{centering}

\begin{theorem}\label{Th:biinequ}
Given $\mu$ and $\nu$ two
Beltrami differentials in $Bel(S)$,  then we have,
\begin{equation}\label{Eq:binary1}\liminf_{t\to 0^+}\frac{d_{AT}([[t\mu]], [[t\nu]])}{t}
\geq\sup_{\qds}\limsup_{n\to\infty} Re\iint_S(\mu-\nu)\phi_n \, dxdy,
\end{equation}
where the first  $supremum$ is taken when  $\{\phi_n\}$ varies over
$\qds$.
\end{theorem}
We have an important corollary.
\begin{cor}\label{Th:munuinf}
Let $\mu$ and $\nu$ be two asymptotically extremal Beltrami
differentials in $\mut$. If the two geodesics $[[t\mu]]$ and
$[[t\nu]]$ ($0\leq t\leq 1$) coincide, then
\begin{equation*}
\sup_{\qds}\limsup_{n\to\infty} Re\iint_S(\mu-\nu)\phi_n \, dxdy=0;
\end{equation*}
in other words, $\mu$ and $\nu$ are infinitesimally asymptotically equivalent.
\end{cor}

The corollary is equivalent to Theorem 5.1 in \cite{Fan}.

To prove  Theorem \ref{Th:biinequ}, we need the  asymptotic
fundamental inequality \cite{EGL,GL}, which is an asymptotic
analogue of the well known Reich-Strebel inequality \cite{RS3}.

\noindent\textbf{The Asymptotic  Fundamental Inequality.} Suppose $f$ is a \qc mapping
from $S$ to $S^\mu$ with $\mu$ its Beltrami differential. Let $H=H(\mut)$.
Then

\begin{flushleft}\begin{equation}\label{Eq:asymineq}
\frac{1}{H}\leq
\liminf_{n\to\infty}\iint_S\frac{\left|1-\mu\frac{\phi_n}{|\phi_n|}\right|^2}{1-|\mu|^2}|\phi_n |\, dxdy,
\end{equation}
\end{flushleft}
for all degenerating sequences $\{\phi_n\}\in \qds$.

\textbf{Proof of Theorem \ref{Th:biinequ}}.
Regard $S$ as $\md/\Gamma$, where $\Gamma$ is a Fuchsian group. Let $\mu$ and
$\nu$ be two Beltrami differentials in $Bel(S)$. For each $t>0$
sufficiently close to zero, there exist two Riemann surfaces $S_t$,
$R_t$ and two \qc mappings $f_t=f^{t\mu}: S\to R_t$,  $g_t=f^{t\nu}:
S\to S_t$, such that the Beltrami differentials of $f_t$ and $g_t$
are $t\mu$ and $t\nu$, respectively.  Suppose $G_t:\md\to\md$ is the
lift of $g_t$ with the points 1,  $i$ and -1 fixed. We can write
$S_t=\md/\Gamma_t$, where
\begin{equation*}
\Gamma_t=\{G_t\circ \gamma\circ G_t^{-1}|\gamma\in \Gamma\}
\end{equation*}
is a Fuchsian group.

Let $\Omega$ be a fundamental domain of $S$. Then
$\Omega_t=G_t(\Omega)$ is a fundamental domain of $S_t$.

Let $\vp$ be an   element of $\q1s$ and $\wt\vp(z)dz^2$ be the lift
of $\vp$. Then $\wt\vp$ satisfies
\begin{equation*}
\wt\vp(\gamma(z))[\gamma'(z)]^2=\wt\vp(z), \ \gamma \in \Gamma,\ \
z\in \md,\end{equation*} and
\begin{equation*}\label{Eq:equiv0}
\iint_{\gamma(\Omega)}|\wt \vp|\,dxdy\equiv1
\end{equation*} for all $\gamma\in \Gamma$.  There is a holomorphic quadratic differential
$\psi(z)dz^2\in Q(\md)$ such that the \po\ series of $\psi$ (see
\cite{Gar},  Chapter 4,  Theorem 3)
\begin{equation}\label{Eq:ser}
\Theta\psi(z)=\sum_{\gamma\in \Gamma}\psi(\gamma(z))[\gamma'(z)]^2
\end{equation} is equal to
$\wt\vp$. We define
\begin{equation*}
\wt\phi_t(z)=\sum_{\gamma_t\in
\Gamma_t}\psi(\gamma_t(z))[\gamma_t'(z)]^2.\end{equation*} Putting
\begin{equation}\label{Eq:ser1}
\wt\vp_t=\frac{\wt\phi_t}{\iint_{\Omega_t}|\wt\phi_t|dxdy},
\end{equation}
we have
\begin{equation*}
\wt\vp_t(\gamma_t(z))[\gamma_t'(z)]^2=\wt\vp_t(z),  \  z\in \md,
\end{equation*}   and \begin{equation}\label{Eq:equiv}
\iint_{{\gamma_t}(\Omega_t)}|\wt \vp_t|\, dxdy\equiv1
\end{equation}
for all  $\gamma_t\in\Gamma_t$, respectively.
 This means that
$\wt\vp_t(z)dz^2$ is a lift of a holomorphic differential $\vp_t\in
Q^1(S_t)$.

Let $\Lambda_t$ be the composition of $f_t$ and $ g_t^{-1}$,   i.e.,
$\Lambda_t=f_t\circ g_t^{-1}:S_t\to  R_t$. Denote by $\lambda_t$ the
complex dilatation of $\Lambda_t$. Let $\wt \mu$,  $\wt\nu$ and
$\wt\lambda_t$ be the lifts of $ \mu$,  $\nu$ and $\lambda_t$,
respectively. Then we have
\begin{equation*}
\wt\lambda_t(w)=\biggl[\frac{t(\wt\mu-\wt\nu)}{1-t^2\wt\mu\ov{\wt
\nu}}\cdot\frac{\pa_zG_t}{\ov{\pa_zG_t}}\biggr]\circ G_t^{-1}(w).
\end{equation*}

Let $H(t)=H([[\lambda_t]])$. Put \[h(t)=\frac{H(t)-1}{H(t)+1}.\]
We can find  a degenerating sequence $\{\vp^n\}$ in $ \qds$, such that
\begin{equation}\label{Eq:biequ}\lim_{n\to\infty} Re\iint_S(\mu-\nu)\vp^n \, dxdy
=\sup_{\qds}\limsup_{n\to\infty} Re\iint_S(\mu-\nu)\phi_n \, dxdy.
\end{equation}

Then every $\vp^n$ corresponds to a $\psi_n\in Q(\de)$. Let
$\wt\vp^n_t$ be given by (\ref{Eq:ser1}). The sequence $\{\wt
\vp_t^n\}$ is a lift of a sequence $\{ \vp_t^n\}$ in $ Q^1(S_t)$
which is  degenerating on $S_t$.

The Asymptotic Fundamental Inequality (\ref{Eq:asymineq}) implies
 \begin{equation*}
\frac{1}{H(t)}=\frac{1-h(t)}{1+h(t)}\leq
\liminf_{n\to\infty}\iint_{S_t}\frac{\left|1-\lambda_t\frac{\vp^n_t}{|\vp^n_t|}\right|^2}{1-|\lambda_t|^2}|\vp^n_t
|dudv,
\end{equation*}
which yields
\begin{equation*} h(t)\geq t
\limsup_{n\to\infty}Re\iint_{\Omega_t} \biggl[(\wt\mu-\wt\nu)
\cdot\frac{\pa_zG_t}{\ov{\pa_zG_t}}\biggr]\circ
G_t^{-1}(w)\cdot\wt\vp^n_t(w)dudv+O(t^2),
\end{equation*}
equivalently,
 \begin{equation*}h(t)\geq t
\limsup_{n\to\infty}Re\iint_{\Omega } (\wt\mu-\wt\nu)
\cdot{(\pa_zG_t)}^2|1-t\wt\nu|^2\cdot \wt\vp^n_t(G_t)\, dxdy+O(t^2).
\end{equation*}
 Therefore,
\begin{equation*} \frac{h(t)}{t}\geq
\limsup_{n\to\infty}Re\iint_{\Omega } (\wt\mu-\wt\nu)
\cdot{(\pa_zG_t)}^2\cdot \wt\vp^n_t(G_t)\, dxdy+O(t).
\end{equation*}
Furthermore, we have
\begin{align*} \liminf_{t\to0}\frac{h(t)}{t}&\geq\liminf_{t\to0}
\limsup_{n\to\infty}Re\iint_{\Omega } (\wt\mu-\wt\nu)
\cdot{(\pa_zG_t)}^2\cdot \wt\vp^n_t(G_t)\, dxdy\\
&\geq
\limsup_{n\to\infty}\liminf_{t\to 0}Re\iint_{\Omega } (\wt\mu-\wt\nu)
\cdot{(\pa_zG_t)}^2\cdot \wt\vp^n_t(G_t)\, dxdy.
\end{align*}
We now show that for every $n$,
\begin{align}\label{Eq:liminft} \liminf_{t\to 0}Re\iint_{\Omega } (\wt\mu-\wt\nu)
\cdot{(\pa_zG_t)}^2\cdot \wt\vp^n_t(G_t)\, dxdy=
Re\iint_{\Omega } (\wt\mu-\wt\nu)
\wt\vp^n(z)\, dxdy.
\end{align}
We may assume that there is a sequence $\{t_m: t_m\to 0\}$ such that
\begin{align*} &\liminf_{t\to 0}Re\iint_{\Omega } (\wt\mu-\wt\nu)
\cdot{(\pa_zG_t)}^2\cdot \wt\vp^n_t(G_t)\, dxdy\\=&
\lim_{m\to \infty}Re\iint_{\Omega } (\wt\mu-\wt\nu)
\cdot{(\pa_zG_{t_m})}^2\cdot \wt\vp^n_{t_m}(G_{t_m})\, dxdy.
\end{align*}
The fact that the Beltrami differential of $G_t$ is $t\wt\nu$
implies that $G_t$ is a good approximation of the identity map $id$
on $\md$, and hence there exists a subsequence of $\{t_m\}$, say,
itself, such that $G_{t_m}$ converges to $id$ uniformly on $\ov\md$
and $\pa_z G_{t_m}(z)\to 1$ almost everywhere in $\md$ (see
\cite{LV}). Thus, ${(\pa_zG_{t_m})}^2\to1$ a.e. in $\md$. The
fundamental domain $\Omega_{t_m}=G_{t_m}(\Omega)$ converges to
$\Omega$. $\wt\vp^n_{t_m}$ converges to $\wt\vp^n$ uniformly on
compact subset of $\md$ passing to a subsequence if necessary.
(\ref{Eq:liminft}) follows readily. Thus,
\begin{equation*}\liminf_{t\to0}\frac{h(t)}{t}\geq\lim_{n\to\infty} Re\iint_S(\mu-\nu)\vp^n \, dxdy,
\end{equation*}
In terms of (\ref{Eq:biequ}),  (\ref{Eq:binary1}) derives immediately. Theorem \ref{Th:biinequ} is proved.

\begin{centering}\section{\!\!\!\!\!{. }Geodesic disks in the universal asymptotic \T space}\label{S:disk}\end{centering}

 $\mut$ (or
$\azmu$) is called a substantial point in $\atde $  (or $AZ(\de)$), if every
 $p\in \pa \de$ is a(n) (infinitesimal) substantial boundary point for $\mut$ (or
$\azmu$); otherwise, $\mut$ (or $\azmu$) is called a non-substantial point.

Let $SP$ and $ISP$
 denote the collection of all (infinitesimal)
substantial points in $\atde$ and $AZ(\de)$, respectively. It is clear that $\atde\backslash SP$
and $AZ(\de)\backslash ISP$ are open and dense in
 $\atde$ and $AZ(\de)$,
respectively.

Let $d_H(z_1,z_2)$ denote the hyperbolic distance between two points
$z_1$, $z_2$ in  $\de$, i.e.,
\begin{equation*}
d_H(z_1,z_2)=\frac{1}{2}\log\frac{1+|\frac{z_1-z_2}{1-\bar
z_1z_2}|} {1-|\frac{z_1-z_2}{1-\bar z_1z_2}|}.
\end{equation*}

Recall that a geodesic disk in $\ats$ is the image of a map $\Psi: \de\hookrightarrow AT(S)$ which
is an isometric embedding with respect to the hyperbolic metric on
$\de$ and the  asymptotic \T metric on $AT(S)$, respectively.

We say that $\mu$ is a non-Strebel extreml if  it is an extremal representative in the non-Strebel point $\emu$ (or $\emu_Z$).
Let  $\mu$ be a non-Strebel extremal with
$\|\mu\|_\infty=k\in (0,1)$.   Then the embedding
\begin{align*}\Psi_\mu:\ &\md\hookrightarrow \ats, \\
&t\longmapsto [[t\mu/k]],
\end{align*}
is a holomorphic isometry.  Because the subsequent Lemma \ref{Th:nonstrebel} indicates that a non-Strebel extremal representative in $\mut$ always exists,  there is at least a geodesic
disk containing $\oo$ and $\mut$.

To obtain Theorem \ref{Th:disks}, we need a series of lemmas.
\begin{lemma}\label{Th:nonstrebel}

Let $\mu\in Bel (S)$. Then,
\\
(1) if $\mu\in M(S)$, then there exists a Beltrami differential
$\nu\in\mut$ such that $\nu$ is a non-Strebel extremal;\\
(2) there exists a Beltrami differential $\nu\in\azmu$ such that
$\nu$ is a non-Strebel extremal.
\end{lemma}
\begin{proof}
We only show the first part (1).

By Theorem 2  on page 296 of \cite{GL}, there is an asymptotic extremal representative in $\mut$, say $\mu$, such that $h(\mut)=h^*(\mu)$.
If $h^*(\mu)=0$, let $\nu$ be identically zero. If $h^*(\mu)>0$, put
\begin{equation*}\nu(z)=
\begin{cases}
\mu(z),& |\mu(z)|\leq h^*(\mu),\\
h^*(\mu)\mu(z)/|\mu(z)|,& |\mu(z)|> h^*(\mu).
\end{cases}
\end{equation*}
In either case, it is easy to verify that $\nu\in \mut$  and is a non-Strebel extremal.
\end{proof}

\begin{lemma}\label{Th:asynu}
Let $\mu\in Bel (\de)$ and $p\in \pa \de$. Then,
\\
(1) if $\mu\in M(\de)$, then there exists a Beltrami differential
$\nu\in\mut$ such that $\nu$ is a non-Strebel extremal  and
$h^*_p(\nu)=h_p(\mut)$;\\
(2) there exists a Beltrami differential $\nu\in\azmu$ such that
$\nu$ is a non-Strebel extremal and $b^*_p(\nu)= b_p(\azmu)$.
\end{lemma}
\begin{proof}We also only show the first part (1).

Case 1. $h_p(\mut)=h(\mut):=h$.

The case  has a trivial proof since any non-Strebel extremal representative in $\mut$ has the required property.

Case 2. $h_p(\mut)<h(\mut):=h$.

 By the definition of
boundary dilatation, there exists a Beltrami differential $\chi\in
\mut$ such that  $h^*_p(\chi)<\min\{ h_p(\mut)+\frac{1}{2},h\}$.

Let $E_n =\{z\in \de:\; |z-p|<\frac{1}{n}\}$. Then  $E_n$
 converges to $\emptyset$ as $n\to \infty$.

There is a number $n_1\geq 1$ such that  $\|\chi\|_\infty<
\min\{h_p(\mut)+\frac{1}{2},h\}$ restricted on $E_{n_1}$.

 Restrict $\chi$ on $\de\backslash E_{n_1+1}$ and
regard $[[\chi]]$ as a point in $AT(\de\backslash
 E_{n_1+1})$. Then $h([[\chi]])=h$. By Lemma \ref{Th:nonstrebel}, we can choose a non-Strebel extremal in $[[\chi]]$,
say $\chi_1(z)$. Define
\begin{equation*}
\nu_1(z)=
\begin{cases}
\chi_1(z),& z\in \de\backslash  E_{n_1+1},\\
\chi(z), &z\in   E_{n_1+1}.\end{cases}
\end{equation*}
Then, $\nu_1$ is a non-Strebel extremal in $\mut$ and $|\nu(z)|\leq
\min\{h_p(\mut)+\frac{1}{2},h\}$ in $E_{n_1+1}$ almost everywhere.

Consider $\nu_1(z)$ on $E_{n_1+1}$. We still have
$h^*_p(\nu_1)=h_p(\mut)$. By the same reason, there is a number
$n_2>n_1+1$, and a Beltrami differential $\nu_2(z)\in M(E_{n_1+1})$
such that $\nu_2\in [[\nu_1]]$ regarded as a
point in $AT(E_{n_1+1})$, $\|\nu_2\|_\infty<
\min\{h_p(\mut)+\frac{1}{2},h\}$ and
$|\nu_2(z)|<\min\{h_p(\mut)+\frac{1}{2^2},h\}$ in $E_{n_2+1}$ almost
everywhere.

Following the construction, we get a sequence $n_{j+1}$ ($j\geq1$),
$n_{j+1}\to \infty$ ($j\to \infty$) and a sequence $\nu_{j+1}(z)\in
M(E_{n_{j}+1})$ such that $\nu_{j+1}\in [[\nu_j]]$ regarded as a
point in $AT(E_{n_{j}+1})$, $\|\nu_{j+1}\|_\infty<
\min\{h_p(\mut)+\frac{1}{2^{j}},h\}$ and
$|\nu_{j+1}(z)|<\min\{h_p(\mut)+\frac{1}{2^{j+1}},h\}$ in
$E_{n_{j+1}+1}$ almost everywhere.

Define
\begin{equation*}
\nu(z)=
\begin{cases}
\nu_1(z),& z\in \de\backslash E_{n_1+1},\\
\nu_2(z), &z\in   E_{n_1+1}\backslash E_{n_2+1},\\
\;\;\,\vdots\\
\nu_j(z), &z\in   E_{n_{j-1}+1}\backslash E_{n_j+1},\\
\;\;\,\vdots\
\end{cases}
\end{equation*}
Then $\nu$ belongs to $\mut$ and is the desired non-Strebel
extremal.
\end{proof}

\begin{lemma}\label{Th:dist}
Let $t_1$ and $t_2$ be two complex numbers and $k_1$ and $k_2$ be
two real numbers. Then we have
\begin{equation}\label{Eq:dist}
\left|\frac{(t_1-t_2)k_1}{1-\ov{t_2}t_1k_1^2}\right|\leq\left|\frac{(t_1-t_2)k_2}{1-\ov{t_2}t_1k_2^2}\right|,\,
\text{ if }\;0<k_1\leq k_2\,\text{ and }\; k_2^2|t_1t_2|<1.
\end{equation}
\end{lemma}
\begin{proof}
Without any loss of generality, we may assume that $t_1t_2\neq0$.
Let $k$ be a real variable and put
\begin{equation*}
F(k)=\left|\frac{(t_1-t_2)k}{1-\ov{t_2}t_1k^2}\right|^2=
\frac{|t_1-t_2|^2k^2}{1+|t_1t_2|^2k^4-2k^2Re(\ov{t_2}t_1)}.
\end{equation*}
It is easy to verify that $F'(k)\geq0$ as $k\in
(0,1/\sqrt{|t_1t_2|})$. Therefore $F(k)$ is an increasing function
on $(0,1/\sqrt{|t_1t_2|})$ and hence (\ref{Eq:dist}) holds.
\end{proof}

\begin{lemma}\label{Th:ret}
Given $k\in (0,1]$ and $s,\,t\in\de$, then
\begin{equation}
\biggl|\frac{(Re(s)-Re(t))k}{1-Re(s)\overline{Re(t)}k^2}\biggr|\leq
\biggl|\frac{t-s}{1-s\bar t}\biggr|.
\end{equation}
\end{lemma}
\begin{proof}
By Lemma \ref{Th:dist}, we have
\begin{equation*}
\biggl|\frac{(Re(s)-Re(t))k}{1-Re(s)\overline{Re(t)}k^2}\biggr|\leq
\biggl|\frac{Re(s)-Re(t)}{1-Re(s)\overline{Re(t)}}\biggr|.
\end{equation*}
On the other hand,
  Lemma 6.4 on page 75 of \cite{GR} indicates that
\begin{equation*}d_H(Re(t),Re(s))=d_H(\frac{t+\bar t}{2},\frac{s+\bar s}{2})\leq
d_H(t,s),
\end{equation*} which implies that
\begin{equation*}
\biggl|\frac{Re(s)-Re(t)}{1-Re(s)\overline{Re(t)}}\biggr|\leq
\biggl|\frac{t-s}{1-s\bar t}\biggr|.
\end{equation*}
\end{proof}

\textbf{Proof of Theorem \ref{Th:disks}.}  It suffices to prove that for any $\mut$ ($\neq \oo$) in $\atde$, there
are infinitely many geodesic disks passing through $\mut$ and $\oo$.
Choose a non-Strebel extremal representative in $\mut$, say $\mu$.
 Then $k_0(\emu)=h(\mut)=h^*(\mu):=h$.

\textit{Case 1.} $\mut$ is not a substantial point.

There is a point $q\in
\pa \de $ which is  not a substantial boundary point for $\mut$. By Lemma
\ref{Th:asynu}, we may assume  $h^*_q(\mu)<h$ in addition.

 By  the definition of boundary dilatation,
we can find a small neighborhood $B(q)$  of $q$ in $\de$ such that
$|\mu(z)|\leq \rho<h$ for some $\rho>0$ in $B(q)$ almost everywhere.
Therefore for any
 $\zeta\in \pa \de\cap  \pa B(q)$, $h_\zeta^*(\mu)\leq \rho$.

Choose  $\delta(z)\in
M(\de)$
 such that  $\|\delta\|_\infty\leq \beta<h-\rho$ and $\delta(z)=0$ when $z\in \de\backslash B(q)$.

Let $\Sigma$ be the collection of the complex-valued functions $\sigma$ defined on $\de$ with the following conditions:
\\
(A)  $\sigma$ is  continuous with  $\sigma(0)=0$ and $\sigma(h)=0$,\\
(B) $\left| \frac{|(s-t)\mu(z)/h|+|\sigma(s)-\sigma(t)|\beta}{1-[s\mu(z)/h+\sigma(s)\delta(z)][\overline{t\mu(z)/h+\sigma(t)\delta(z)}]}\right|
\leq\left|\frac{s-t}{1-s\bar t}\right|$,\;
$\; t,\;s\in \de$, $z\in B(q)$.

We claim that $\Sigma$ contains uncountably many  elements. At first, let $\sigma$ be a Lipschitz continuous function on
$\de$ with the following conditions,\\
(i) for some small $\alpha>0$, $|\sigma(s)-\sigma(t)|<\alpha |s-t|$, $t,\;s\in \de$,\\
(ii) $\sigma(0)=0$ and $\sigma(h)=0$,\\
(iii) for some small $t_0$ in $(0,h)$,  $\sigma(t)\equiv0$ when $|t|\geq t_0$,

Secondly, we show that when $t_0$ and $\alpha$ are sufficiently small, $\sigma$ belongs to $\Sigma$, for which it suffices to show that $\sigma$ satisfies the condition (B). Let $t,\;s\in \de$. It is no harm to assume that $|t|\leq |s|$.

Case 1. $|t|\geq t_0$.

Since $\sigma(s)=\sigma(t)=0$, by Lemma \ref{Th:dist}, we have
\begin{align*}
&\left|\frac{|(s-t)\mu(z)/h|+|\sigma(t)-\sigma(s)|\beta}{1-[s\mu(z)/h+\sigma(s)\delta(z)][\overline{t\mu(z)/h+\sigma(t)\delta(z)}]}\right|
=\left|\frac{(s-t)\mu(z)/h}{1-[s\mu(z)/h][\overline{t\mu(z)/h}]}\right|\\
&\leq\left|\frac{s-t}{1-s\bar t}\right|,  \; z\in B(q).
\end{align*}

Case 2. $|t|< t_0$.

Put $\gamma=\rho/h+\alpha\beta$ and choose small $\alpha>0$ such that  $\gamma<1$.  On the one hand, since $|\sigma(t)|\leq \alpha |t|$ and $|\sigma(s)|\leq \alpha |s|$, it holds that
\begin{align*}
&\left|\frac{|(s-t)\mu(z)/h|+|\sigma(t)-\sigma(s)|\beta}{1-[s\mu(z)/h+\sigma(s)\delta(z)][\overline{t\mu(z)/h+\sigma(t)\delta(z)}]}\right|
\leq \left|\frac{(s-t)(\rho/h+\alpha\beta)}{1-[|s|\rho/h+\alpha|s|\beta][|t|\rho/h+\alpha |t|\beta]}\right|\\
&\leq \left|\frac{(s-t)(\rho/h+\alpha\beta)}{1-[\rho/h+\alpha\beta][t_0(\rho/h+\alpha\beta)]}\right|
=\gamma\left|\frac{s-t}{1-t_0\gamma^2}\right|,\; z\in B(q).
\end{align*}
On the other hand, we have
\begin{align*}
\left|\frac{s-t}{1-s\bar t}\right|\geq\left|\frac{s-t}{1+t_0}\right|.
\end{align*}
When $t_0$ is sufficiently small, we can get
\begin{align*}
\left|\frac{s-t}{1+t_0}\right|\geq \gamma\left|\frac{s-t}{1-t_0\gamma^2}\right|.
\end{align*}
Therefore, when $t_0$ and $\alpha$ are sufficiently small, $\sigma$ satisfies the condition (B).

For a given $\sigma\in \Sigma$, define for $t\in \de$,
 \begin{equation}\label{Eq:nonsub}
 \mu_t(z)=
 \begin{cases}
t\mu(z)/h,&z\in \de\backslash B(q),\\
t\mu(z)/h+\sigma(t)\delta(z),&z\in  B(q),
 \end{cases}
 \end{equation}
and the map
\begin{align*}\Psi_{\sigma}:\ &\md \hookrightarrow \atde, \\
&t\longmapsto [[\mu_t]].
\end{align*}
It is obvious that $\Psi_{\sigma}(\de)$ contains $\mut$ and the basepoint $\oo$.  We show that  $\Psi_{\sigma}$ is an isometric
embedding.
 It is sufficient to verify
that
\begin{equation}\label{Eq:sst}
d_{AT}([[\mu_t]],[[\mu_s]])=d_H(t,s),\; t,\,s\in \de,
\end{equation}
where $d_H$ expresses the hyperbolic distance on the unit disk.

Let $f_s:\, \de\to \de$ and $f_t:\, \de\to \de$ be \qc
mappings with Beltrami differentials $\mu_s$ and $\mu_t$
respectively. It is convenient to assume that $t\neq 0$ and $s\neq t$. Set $F_{s,t}=f_s\circ f_t^{-1}$ and assume that the
Beltrami differential of $F_{s,t}$ is $\nu_{s,t}$. Then a simple
computation shows,
\begin{equation*}
\nu_{s,t}\circ
f_t(z)=\frac{1}{\tau}\frac{\mu_s(z)-\mu_t(z)}{1-\ov{\mu_t(z)}\mu_s(z)},
\end{equation*}
where $z=f_t^{-1}(w)$ for $w\in  \de$ and $\tau=\ov{\partial
f_t}/\partial f_t$. We have

\begin{equation}\label{Eq:ft1}
\nu_{s,t}\circ
f_t(z)=\begin{cases}
\frac{1}{\tau}\frac{s-t}{1-  s\bar t|\mu(z)|^2/h^2 }\frac{\mu(z)}{h},\,&z\in \de\backslash B(q),\\
\frac{1}{\tau}\frac{(s-t)\mu(z)/h+[\sigma(s)-\sigma(t)]\delta(z)}{1-[s\mu(z)/h+\sigma(s)\delta(z)]
\overline{t\mu(z)/h+\sigma(t)\delta(z)}},\, &z\in B(q).
\end{cases}
\end{equation}
Since $\sigma\in \Sigma$, due to condition (B) we see that
restricted on $f_t(B(q))$,
\begin{equation}\label{Eq:vvst}
\|\nu_{s,t}\|_\infty\leq\left|\frac{s-t}{1-s\bar  t}\right|.
\end{equation}

Suppose $p\in \pa \de$ is a substantial boundary point for $\mut$. By
Lemma \ref{Th:atslem1} there is a degenerating Hamilton sequence
$\{\psi_n\}\subset Q^1(\de)$ towards $p$ such that
\begin{align*}
h=\limn\iint_\de \mu(z)\psi_n(z)dxdy.
\end{align*}
Then, we have
\begin{align*}\label{Eq:qs11}
|t|=\limn\iint_\de \mu_t(z)e^{-iarg t}\psi_n(z)dxdy.
\end{align*}
On the other hand, it is easy to see that
$h([[\mu_t]])=h^*(\mu_t)=|t|$ and hence $\mu_t$ is an asymptotic
extremal.
   Therefore, the Beltrami differential $\wt \mu_t$ of  $f^{-1}_t$ is also an asymptotic extremal
where ${\wt\mu}_t=-\mu_t({f}_t^{-1})\ov{\partial
{f}_t^{-1}}/\partial {f}_t^{-1}$. $f_t(p)$ is a
substantial boundary point for $[[\wt\mu_t]]$ and there is a
degenerating Hamilton sequence $\{\wt\psi_n\}\subset Q^1(\de)$
towards $f_t(p)$ such that
\begin{equation*}
\lim_{n\to\infty}Re\iint_{\de} \wt\mu_t\wt\psi_n(w)dudv=h([[\wt\mu_t]])=|t|.
\end{equation*}
 Furthermore,
\begin{equation}\label{Eq:vst}\begin{split}
&\lim_{n\to\infty}Re\iint_{\de}
\nu_{s,t}(w)e^{i(\theta+\arg t)}\wt\psi_n(w)dudv\\
&=\lim_{n\to\infty}Re\iint_{\de}\frac{s-t}{1-s\bar  t}
\frac{{\wt\mu}_t}{t}e^{i(\theta+\arg t)}\wt\psi_n(w)dudv
=\left|\frac{s-t}{1-s\bar  t}\right|,
\end{split}
\end{equation}
where $\theta=-\arg \frac{s-t}{1-s\bar  t}$.

 Thus, by (\ref{Eq:vvst}),
(\ref{Eq:vst}) and Lemma \ref{Th:atslem1}, it follows  that
$h([[\nu_{s,t}]])=\left|\frac{s-t}{1-s\bar  t}\right|$, $\nu_{s,t}$ is asymptotically
extremal and the equality (\ref{Eq:sst}) holds.

It remains to show that there are infinitely many geodesic disks passing through $\mut$ and $[[0]]$
when  $\sigma$ varies over $\Sigma$ and $\delta(z)$ varies
over $ M(\de)$ suitably, respectively.

Firstly, choose $\delta(z)$  in $M(\de)$ such that \begin{equation}\label{Eq:lakic}
\sup_{\qdde}\limsup_{n\to\infty} Re\iint_\de\delta\vp_n \, dxdy=c>0,
\end{equation}
where the  $supremum$ is over all sequences   $\{\vp_n\}$ in $\qdde$
degenerating towards $q$.

Secondly,  we  choose small $t_0$ in $(0,h)$, small $\alpha>0$   and $\sigma\in \Sigma$ such that $\sigma(t)\equiv0$ whenever $|t|\geq t_0$  and
 $\sigma(t)=\alpha t$  when $t\in [0,t_0/2]$.

\textit{Claim.} When  $\alpha$ varies in a small range,
the   geodesic disks $\Psi_{\sigma}(\de)$  are mutually different.

Let $\alpha_1$ and $\alpha_2$ be two small different positive numbers and
$\sigma_j(t)=\alpha_j t$  when $t\in [0,t_0]$ ($j=1,2$), respectively.
Now, the corresponding expression of equation (\ref{Eq:nonsub}) is
 \begin{equation*}\mu_t^j(z)=
 \begin{cases}
t\mu(z)/h,&z\in \de\backslash B(q),\\
t\mu(z)/h+\sigma_j(t)\delta(z),&z\in B(q),\; j=1,2.
 \end{cases}
 \end{equation*}

When $t\in [0,h]$, they correspond to
geodesics $G_j=\{[[\mu^j_t]]:t\in [0,h]\}$ ($j=1,2$), respectively.
Note that when $t\in [0,t_0/2]$,
 \begin{equation*}\mu_t^j(z)=
 \begin{cases}
t\mu(z)/h,&z\in \de\backslash B(q),\\
t\mu(z)/h+t\alpha_j\delta(z),&z\in B(q),\; j=1,2.
 \end{cases}
 \end{equation*}
 Define \begin{equation*}\mu^j(z)=
 \begin{cases}
\mu(z)/h,&z\in \de\backslash B(q),\\
\mu(z)/h+\alpha_j\delta(z),&z\in B(q),\; j=1,2.
 \end{cases}
 \end{equation*}

Since
\begin{align*}\sup_{\qdde}&\limsup_{n\to\infty} Re\iint_\de(\mu^1-\mu^2)\vp_n \, dxdy=
\sup_{\qdde}\limsup_{n\to\infty} Re\iint_\de(\alpha_1-\alpha_2)\delta\vp_n \, dxdy\\
\geq &|\alpha_1-\alpha_2|c>0,
\end{align*}
by Theorem \ref{Th:biinequ}, the
geodesics $G_1$ and $G_2$ are obviously different and hence the geodesic disks $\Psi_{\sigma_1}(\de)$ and $\Psi_{\sigma_2}(\de)$   are  different.

 If fix small $\alpha>0$ and let $\delta$ vary suitably in $M(\de)$, then we can also get infinitely many geodesic disks as desired.

\textit{Case 2.} $\mut$ is a substantial point.

Fix a boundary point
$p\in \pa \de$. Let  $B(p)=\{z\in \de:\;|z-p|<r$\} for small $r>0$ and  $E=\de\backslash B(p)$. Define for $t\in\de$
\begin{equation}\label{Eq:geodisk1}
\mu_t(z):=\begin{cases}
t\mu(z)/h,&z\in B(p),\\
Re(t)\mu(z)/h,& z\in E.
\end{cases}
\end{equation}
and the map
\begin{align*}\Psi_{E}:\ &\de \hookrightarrow \atde, \\
&t\longmapsto [[\mu_t]].
\end{align*}
We claim that  $\Psi_{E}$ is an isometric
embedding.
It suffices to check the equality:
\begin{equation}\label{Eq:stdat}
\dat([[\mu_t]],[[\mu_s]])=d_H(t,s),\;t,\,s\in \de.\end{equation}
 Using  the previous
notation, we have
\begin{equation}\label{Eq:atsft1}
\nu_{s,t}\circ f_t(z)=\begin{cases}
\frac{1}{\tau}\frac{s-t}{1-  s\bar t|\mu(z)|^2/h^2 }\frac{\mu(z)}{h},\,&z\in B(p),\\
\frac{1}{\tau}\frac{Re(s)-Re(t)}{1-
Re(s)\overline{Re(t)}|\mu(z)|^2/h^2}\frac{\mu(z)}{h},\, &z\in E.
\end{cases}
\end{equation}
Firstly, it derives readily that restricted on $f_t(B(p))$,
\begin{equation}\label{Eq:vvst1}
\|\nu_{s,t}\|_\infty=\biggl|\frac{s-t}{1-s\bar t}\biggr|.
\end{equation}
Secondly, Lemma \ref{Th:ret} indicates that restricted on $f_t(E)$,
\begin{equation*}
\|\nu_{s,t}\|_\infty\leq\biggl|\frac{s-t}{1-s\bar t}\biggr|.
\end{equation*}
 By a similar argument to Case 1, we can get
 \begin{equation*}
h([[\nu_{s,t}]])=k_0([\nu_{s,t}])=\biggl|\frac{s-t}{1-s\bar
t}\biggr|.
\end{equation*}
(\ref{Eq:stdat}) follows immediately.

It remains to show that when $r$ varies in a suitable range, equivalently, when $E$ varies, the geodesic disks $\Psi_E(\de)$ are mutually different. Fix $t=\lambda i$ where $\lambda\in (0,1).$
Since $Re(t)=0$, $\mu_t(z)=0$ on $E$. Therefore, none of inner points in the arc $\pa\de\cap \pa E$  is a substantial boundary points for $[[\mu_t]]$ but  all  points in the arc $\pa\de\cap \pa B(p)$ are  substantial boundary points. Therefore, when $r$ varies, $[[\mu_t]]$ are mutually different. Thus,  we get
infinitely many geodesic disks as required. Actually, in such a case, these geodesic disks  contain the straight line
$\{[[t\mu/h]]:\,t\in (-1,1)\}.$

The proof of two cases above gives the following corollaries respectively.

\begin{cor}\label{Th:geoats1}
Suppose $\mut$ is not a substantial point in $\atde$.  Then there are infinitely many
 geodesics connecting $\mut$ to the
basepoint $\oo$.
\end{cor}

\begin{cor}\label{Th:straight}
Suppose that $\mut$ ($\neq \oo$) is  a substantial point in $\atde$ and $\mu$ is a non-Strebel extremal representative.  Then there are infinitely many
 geodesic disks containing the straight line
$\{[[t\mu/k]]:\,t\in (-1,1)\},$ where $k=\|\mu\|_\infty\in (0,1)$.
\end{cor}

\begin{centering}\section{\!\!\!\!\!{. }Geodesic planes in the tangent space}\label{S:plane}\end{centering}

A geodesic plane in $\azs$ is the image of a map $\Phi:
\mathbb{C}\to \azs$ which is an isometry with respect to the
Euclidean metric on $\mathbb{C}$ and the infinitesimal asymptotic
metric on $\azs$, respectively.

The infinitesimal version of Theorem
\ref{Th:disks} is as follows.
\begin{theorem}\label{Th:infdisks}
In the tangent space $\azde$, there are always  infinitely many
geodesic planes containing   two points.
\end{theorem}
\begin{proof}
It suffices to prove that for any $\azmu$ ($\neq \oaz$) in $\azde$, there
are infinitely many geodesic planes passing through $\azmu$ and $\oaz$.
By Lemma \ref{Th:nonstrebel}, we  can choose a non-Strebel extremal representative in $\azmu$, say $\mu$. Since $\azde$ is a Banach space, without loss of generality it is convenient to assume that
$\|\mu\|_\infty=b^*(\mu)=b(\azmu)=1$.

\textit{Case 1.} $\azmu$ is not an infinitesimal substantial point.

Suppose $q\in
\pa \de $ is not a substantial boundary point for $\azmu$. By Lemma
\ref{Th:asynu}, we may assume  $b^*_q(\mu)<1$ in addition.

 By  the definition of boundary dilatation,
we can find a small neighborhood $B(q)$  of $q$ in $\de$ such that
$|\mu(z)|\leq \rho<1$ for some $\rho>0$ in $B(q)$ almost everywhere.
Therefore for any
 $\zeta\in \pa \de\cap \pa B(q)$, $b_\zeta^*(\mu)\leq \rho$.

 Choose  $\delta(z)\in
M(\de)$
 such that  $\|\delta\|_\infty\leq \beta<1-\rho$ and $\delta(z)=0$ when $z\in \de\backslash B(q)$.

Let $\Sigma'$ be the collection of the complex-valued functions $\sigma$ defined on $\mc$ with the following conditions:
\\
(A)  $\sigma$ is  continuous with  $\sigma(0)=0$ and $\sigma(1)=0$,\\
(B) $|s-t|\rho+|\sigma(t)-\sigma(s)|\beta
\leq|s-t|,$ $t,\;s\in \mc$, $z\in B(q)$.

Since $\rho<1$ and $\beta <1-\rho$,  $\Sigma'$ contains uncountably many  elements. In fact, if $\sigma$ is a Lipschitz continuous function on $\mc$ with the following conditions,\\
(i) for some small $\alpha>0$, $|\sigma(s)-\sigma(t)|<\alpha |s-t|$, $t,\;s\in \mc$,\\
(ii) $\sigma(0)=0$ and $\sigma(1)=0$,\\
(iii)  $\rho+\alpha \beta<1$,\\
then $\sigma\in \Sigma'$.

Given $\sigma\in \Sigma'$,
define for $t\in \mc$,
 \begin{equation}\label{Eq:infnonsub}
 \mu_t(z)=
 \begin{cases}
t\mu(z),&z\in \de\backslash B(q),\\
t\mu(z)+\sigma(t)\delta(z),&z\in  B(q),
 \end{cases}
 \end{equation}
and the map
\begin{align*}\Phi_{\sigma}:\ &\mc \hookrightarrow \azde, \\
&t\longmapsto [[\mu_t]]_{AZ}.
\end{align*}
It is obvious that $\Phi_{\sigma}(\mc)$ contains $\azmu$ and the basepoint $\oaz$.  We show that  $\Phi_{\sigma}$ is an isometric
embedding.
 It is sufficient to verify
that
\begin{equation}\label{Eq:infsst}
\|[[\mu_s-\mu_t]]_{AZ}\|=|s-t|,\; t,\,s\in \mc.
\end{equation}
At first, it is obvious that
\begin{align*}
\|\mu_s-\mu_t\|_\infty=|s-t|,
\end{align*}

Suppose $p\in \pa \de$ is a substantial boundary point for $\azmu$. By
Lemma \ref{Th:atslem1} there is a degenerating Hamilton sequence
$\{\psi_n\}\subset Q^1(\de)$ towards $p$ such that
\begin{align*}
1=\limn\iint_\de \mu(z)\psi_n(z)dxdy.
\end{align*}
Therefore, we have
\begin{align*}
|s-t|=\limn\iint_\de [\mu_s(z)-\mu_t(z)]e^{-iarg(s-t)}\psi_n(z)dxdy, \;s\neq t,
\end{align*}
which implies the equality (\ref{Eq:infsst}).

It remains to show that there are infinitely many geodesic planes passing through $\azmu$ and $\oaz$
when  $\sigma$ varies over $\Sigma'$ and $\delta(z)$ varies
over $ M(\de)$ suitably, respectively.

Choose $\delta(z)$  in $M(\de)$ such that (\ref{Eq:lakic}) holds.
Fix a  small $t_0$ in $(0,1)$.  Choose $\sigma\in \Sigma'$ such that $\sigma(t)\equiv0$ whenever $|t|\geq t_0$  and
 $\sigma(t)=\alpha t$  when $t\in [0,t_0/2]$ where $\alpha>0$ satisfying  $\rho+\alpha \beta<1$.
Note that when $t\in [0,t_0/2]$,
 \begin{equation*}\mu_t(z)=
 \begin{cases}
t\mu(z)/h,&z\in \de\backslash B(q),\\
t\mu(z)/h+t\alpha\delta(z),&z\in B(q).
 \end{cases}
 \end{equation*}

The geodesic planes $\Phi_{\sigma}$ contain the geodesics
 $G_\alpha=\{[[\mu_t]]:t\in [0,1]\}$ respectively.

Due to the equality (\ref{Eq:lakic}),
the   geodesics $G_\alpha$  are mutually different when $\alpha$ varies in a small range. Therefore, the geodesic planes $\Phi_{\sigma}(\mc)$ are mutually different.

If fix small $\alpha>0$ and let $\delta$ vary suitably in $M(\de)$, then we can also get infinitely many geodesic planes as required.

\textit{Case 2.} $\azmu$ is a substantial point.

 Fix a boundary point
$p\in \pa \de$. Let  $B(p)=\{z\in \de:\;|z-p|<r\}$ for small $r>0$ and  $E=\de\backslash B(p)$. Define for $t\in\mc$
\begin{equation}\label{Eq:infgeodisk1}
\mu_t(z):=\begin{cases}
t\mu(z),&z\in B(p),\\
Re(t)\mu(z),& z\in E.
\end{cases}
\end{equation}
and the map
\begin{align*}\Phi_{E}:\ &\mc \hookrightarrow \azde, \\
&t\longmapsto [[\mu_t]]_{AZ}.
\end{align*}
We claim that  $\Phi_{E}$ is an isometric
embedding.

 Note that
\begin{equation}\label{Eq:azsft1}
\mu_s-\mu_t=\begin{cases}
(s-t)\mu(z),\,&z\in B(p),\\
(Re(s)-Re(t))\mu(z),\, &z\in E.
\end{cases}
\end{equation}
It is easy  to check the equality:
\begin{equation}\label{Eq:stdaz}
\|[[\mu_s-\mu_t]]_{AZ}\|=|s-t|,\; t,\,s\in \mc.
\end{equation}

It remains to show that when $r$ varies in a suitable range, the geodesic planes $\Phi_E(\mc)$ are mutually different. Fix $t=\lambda i$ where $\lambda\in (0,1).$
Since $Re(t)=0$, $\mu_t(z)=0$ on $E$. Therefore, none of inner points in the arc $\pa\de\cap \pa E$  is a substantial boundary points for $[[\mu_t]]_{AZ}$ but  all  points in the arc $\pa\de\cap \pa B(p)$ are  substantial boundary points. Therefore, when $r$ varies, $[[\mu_t]]_{AZ}$ are mutually different. Thus,  we get
infinitely many geodesic planes as required. In particular, these geodesic planes  contain the straight line
$\{[[t\mu]]_{AZ}:\,t\in \mr\}.$
\end{proof}
The following two corollaries follow from the proof of two cases above separately.

\begin{cor}\label{Th:infgeoats1}
Suppose $\azmu$ is not a substantial point in $\azde$.  Then there are infinitely many
 geodesics connecting $\azmu$ to the
basepoint $\oaz$.
\end{cor}

\begin{cor}
Suppose that $\azmu$ ($\neq \oaz$) is  a substantial point in $\azde$ and $\mu$ is a non-Strebel extremal representative.  Then there are infinitely many
 geodesic planes containing the straight line
$\{[[t\mu/k]]_{AZ}:\,t\in \mr\},$ where $k=\|\mu\|_\infty>0$.
\end{cor}

\renewcommand\refname{\centerline{\Large{R}\normalsize{EFERENCES}}}
\medskip
\addcontentsline{toc}{section}{\bf\large{R}\normalsize{EFERENCES}}

\noindent \textit{Guowu Yao}\\
Department of Mathematical Sciences\\
  Tsinghua University\\Beijing,  100084,  People's Republic of
  China \\
  E-mail: \texttt{gwyao@math.tsinghua.edu.cn}
\end{document}